\documentclass[11pt]{article}

\usepackage{amsmath,amssymb,amsthm}
\numberwithin{equation}{section}

\newtheorem{theorem}{Theorem}[section]
\newtheorem{lemma}[theorem]{Lemma}
\newtheorem{problem}{Problem}

\title{\bf A note on a problem of Ma and Chen}
\author{Bing-Ling Wu, Wu-Xia Ma\\
	{\small School of Science, Nanjing University of Posts and Telecommunications,}\\
	{\small Nanjing 210023, PR China}\\
	{\small E-mail: wubingling911219@163.com, mawuxia@126.com}}
\date{}

\begin{document}
	
	\maketitle
	
	\begin{abstract}
		Let $S=(s_1,s_2,\cdots)$ be a strictly increasing sequence  of positive integers. For any real number $\alpha>0$, let $S_\alpha=(\lfloor\alpha s_1\rfloor,\lfloor\alpha s_2\rfloor,\cdots)$. We prove that if $s_{n+1}\le 2s_n$ for every positive integer $n$, then there exists $0<\alpha<1$ such that $S_\alpha$ is strongly complete. In addition, this gives a negative answer to a problem of Ma and Chen in [Acta Arith. 203 (2022), 307--318]. We also prove that the same result is true if $s_{n+1}\le 2s_n$ for all sufficiently large $n$.
	\end{abstract}
	
	
	
	\section{Introduction}
	
	Let $\mathbb{N}_0=\mathbb{N}\cup\{0\}$. For any sequence $A=(a_1,a_2,\cdots)$ of nonnegative integers, let
	$$P(A)=\left\{\sum_{i\in F}a_i:F\subseteq\mathbb{N},\ |F|<\infty\right\}.$$
	The sequence $A$ is called complete if all sufficiently large integers belong to $P(A)$. It is called strongly complete if
	$$P(A)=\mathbb{N}_0.$$
	For any sequence $S=(s_1,s_2,\cdots)$ of positive integers and any real number $\alpha>0$, let
	$$S_\alpha=(\lfloor\alpha s_1\rfloor,\lfloor\alpha s_2\rfloor,\cdots),$$
	and let
	$$U_S=\{\alpha>0:S_\alpha\text{ is complete}\},$$
	$$V_S=\{\alpha>0:S_\alpha\text{ is strongly complete}\}.$$
	Hegyv\'ari \cite{Heg} proved that if
	$$\lim_{n\to\infty}(s_{n+1}-s_n)=+\infty,
	\quad s_{n+1}<\gamma s_n\quad(n\ge n_0),$$
	where $1<\gamma<2$, and $U_S\ne\emptyset$, then $\mu(U_S)>0$, where $\mu$ denotes the Lebesgue measure. Chen and Fang \cite{CF} removed the condition
	$$\lim_{n\to\infty}(s_{n+1}-s_n)=+\infty.$$
	In addition, Chen and Fang \cite{CF} proved that $\mu(U_S)>0$ if
	$$s_{n+1}<\gamma s_n\quad(n\ge n_0),
	\quad 1<\gamma\le\frac 74.$$
	Fang and Liu \cite{FL} improved it to
	$$1<\gamma\le\sqrt[4]{13}.$$
	Ma and Chen \cite[Theorem 1.4]{MC} proved that $\mu(V_S)>0$ if
	$$s_{n+1}<\gamma s_n\quad(n\ge n_0),
	\quad 1<\gamma<2.$$
	For related results, see \cite{CT,JM,TX}.
	
	Ding \cite{Ding} proved that the answer to Problem 3.2 of Ma and Chen is positive. Problem 3.1 of Ma and Chen is the following.
	
	\begin{problem}\label{prob1}
		Does there exist a sequence of positive integers
		$$s_1<s_2<\cdots,\quad s_{n+1}\le 2s_n,$$
		such that $P(S_\alpha)\ne\mathbb{N}_0$ for every $0<\alpha<1$?
	\end{problem}
	
	We prove the following theorem.
	
	\begin{theorem}\label{thm1}
		Let $S=(s_1<s_2<\cdots)$ be a sequence of positive integers satisfying
		$$s_{n+1}\le 2s_n\quad(n\ge 1).$$
		Then there exists $0<\alpha<1$ such that
		$$P(S_\alpha)=\mathbb{N}_0.$$
		If $s_1\ge 2$, then $\alpha$ can be chosen to satisfy
		$$\frac 1{s_1}\le\alpha<\frac 2{s_1}.$$
	\end{theorem}
	
	Theorem \ref{thm1} gives a negative answer to Problem \ref{prob1}.
	
	\begin{theorem}\label{thm2}
		Let $S=(s_1<s_2<\cdots)$ be a sequence of positive integers such that
		$$s_{n+1}\le 2s_n$$
		for all sufficiently large $n$. Then there exists $0<\alpha<1$ such that
		$$P(S_\alpha)=\mathbb{N}_0.$$
	\end{theorem}
	
	\section{Proof of Theorems}
	
	We use the following lemma of Brown \cite{Brown}.
	
	\begin{lemma}\label{lem1}
		Let $1=a_1\le a_2\le\cdots$ be positive integers. Then
		$$P((a_1,a_2,\cdots))=\mathbb{N}_0$$
		if and only if
		$$a_j\le 1+\sum_{i<j}a_i\quad(j\ge 2).$$
	\end{lemma}
	
	\begin{proof} Suppose that
		$$a_j>1+\sum_{i<j}a_i$$
		for some $j$. Every sum of terms with indices less than $j$ is at most $\sum_{i<j}a_i$, and every sum containing a term with index at least $j$ is at least $a_j$. Hence
		$$1+\sum_{i<j}a_i\notin P((a_1,a_2,\cdots)).$$
		
		Now suppose that the inequalities in the lemma hold. We use induction on $j$. The result is clear for $j=1$. Suppose that the first $j-1$ terms represent every integer from $0$ to
		$$\sum_{i<j}a_i.$$
		Then the sums containing $a_j$ represent every integer from $a_j$ to
		$$a_j+\sum_{i<j}a_i.$$
		Since $a_j\le 1+\sum_{i<j}a_i$, it follows that the first $j$ terms represent every integer from $0$ to
		$$\sum_{i\le j}a_i.$$
		Since $a_i\ge 1$, the partial sums tend to infinity, so $P((a_1,a_2,\cdots))=\mathbb{N}_0$.
		This completes the proof of Lemma \ref{lem1}.
	\end{proof}
	
	Let $S=(s_1<s_2<\cdots)$ be a sequence of positive integers.
	For any real number $x$, let $\lfloor x\rfloor$ denote the greatest integer not exceeding $x$, and let $\{x\}=x-\lfloor x\rfloor$.
	For $j\ge 1$ and $x>0$, let
	$$H_j(x)=1+\sum_{i<j}\lfloor xs_i\rfloor-\lfloor xs_j\rfloor,
	\quad \Phi_j(x)=2+\sum_{i<j}\lfloor xs_i\rfloor-xs_j.$$
	Then
	\begin{equation}\label{eq1}
		\Phi_j(x)=H_j(x)+1-\{xs_j\}.
	\end{equation}
	Noting that $H_j(x)$ is an integer, we have
	$$\Phi_j(x)>0\quad\Longleftrightarrow\quad H_j(x)\ge 0.$$
	
	\begin{lemma}\label{lem2}
		Let $n\ge 1$ and $1\le j\le n$. If
		$$\Phi_j(x)\ge\frac {s_j}{s_n},$$
		then $\Phi_j(x+u)>0$ for every $0\le u<1/s_n$.
	\end{lemma}
	
	\begin{proof} Since $0\le u<1/s_n$, it follows that
		$$us_j<\frac {s_j}{s_n}\le\Phi_j(x).$$
		For $i<j$, we have $\lfloor(x+u)s_i\rfloor\ge\lfloor xs_i\rfloor$. Thus
		\begin{align*}
			\Phi_j(x+u)&=2+\sum_{i<j}\lfloor(x+u)s_i\rfloor-(x+u)s_j\\
			&\ge\Phi_j(x)-us_j>0.
		\end{align*}
		This completes the proof of Lemma \ref{lem2}.
	\end{proof}
	
	\begin{lemma}\label{lem3}
		Suppose that $s_1\ge 2$ and $s_{n+1}\le 2s_n$ for every $n$. There exist real numbers $\alpha_n$ such that, for every $n\ge 1$,
		\begin{equation}\label{eq2}
			\Phi_j(\alpha_n)\ge\frac {s_j}{s_n}\quad(1\le j\le n),
		\end{equation}
		\begin{equation}\label{eq3}
			\alpha_{n+1}\ge\alpha_n,
			\quad \alpha_{n+1}+\frac 1{s_{n+1}}<\alpha_n+\frac 1{s_n},
		\end{equation}
		and
		\begin{equation}\label{eq4}
			\frac 1{s_1}\le\alpha_n<\alpha_n+\frac 1{s_n}\le\frac 2{s_1}.
		\end{equation}
		Let
		$$I_n=\left[\alpha_n,\alpha_n+\frac 1{s_n}\right).$$
		Then $I_{n+1}\subset I_n$, and for every $x\in I_n$,
		$$\lfloor xs_j\rfloor\le 1+\sum_{i<j}\lfloor xs_i\rfloor\quad(2\le j\le n),$$
		while $\lfloor xs_1\rfloor=1$.
	\end{lemma}
	
	\begin{proof} Let
		$$\alpha_1=\frac 1{s_1}.$$
		Then $\Phi_1(\alpha_1)=1$, and hence \eqref{eq2} holds for $n=1$.
		
		Suppose that $\alpha=\alpha_n$ has been chosen. Write
		$$p=s_n,\quad q=s_{n+1},\quad d=2p-q.$$
		Then $p<q\le 2p$ and $d\ge 0$. Let
		$$a_i=\lfloor\alpha s_i\rfloor,
		\quad H_j=1+\sum_{i<j}a_i-a_j,
		\quad r=\{\alpha p\},
		\quad e=a_{n+1}-2a_n.$$
		By $\alpha p=a_n+r$ and $q=2p-d$, we have
		$$\alpha q=2a_n+2r-\alpha d$$
		and
		\begin{equation}\label{eq5}
			e=\lfloor2r-\alpha d\rfloor\le 1.
		\end{equation}
		In view of \eqref{eq2}, we have
		$$H_n+1-r=\Phi_n(\alpha)\ge 1.$$
		Also,
		\begin{equation}\label{eq6}
			H_{n+1}=1+\sum_{i<n}a_i+a_n-a_{n+1}=H_n-e.
		\end{equation}
		If $r>0$, then $H_n\ge r>0$. Noting that $H_n$ is an integer, we have $H_n\ge 1$. By \eqref{eq5} and \eqref{eq6}, we have $H_{n+1}\ge 0$. If $r=0$, then $e=\lfloor-\alpha d\rfloor\le 0$, and again $H_{n+1}\ge 0$. It follows from \eqref{eq1} that
		$$\Phi_{n+1}(\alpha)>0.$$
		
		If $\Phi_{n+1}(\alpha)\ge 1$, let $\alpha_{n+1}=\alpha$. For $j\le n$, by \eqref{eq2} and $q>p$, we have
		$$\Phi_j(\alpha_{n+1})\ge\frac {s_j}{p}>\frac {s_j}{q}.$$
		Also,
		$$\Phi_{n+1}(\alpha_{n+1})\ge 1=\frac {s_{n+1}}q.$$
		Hence
		$$\alpha_{n+1}+\frac 1q<\alpha_n+\frac 1p.$$
		
		Now we assume that
		$$0<\Phi_{n+1}(\alpha)<1.$$
		If $H_{n+1}\ge 1$, then it follows from \eqref{eq1} that $\Phi_{n+1}(\alpha)>1$, a contradiction. Thus $H_{n+1}=0$. It follows from \eqref{eq1} and $\Phi_{n+1}(\alpha)<1$ that $\{\alpha q\}>0$. Hence
		\begin{equation}\label{eq7}
			H_{n+1}=0,
			\quad \{\alpha q\}>0.
		\end{equation}
		Suppose that $r=0$. Then $H_n\ge 0$ and $e\le 0$. By \eqref{eq6} and \eqref{eq7}, we have $H_n=e=0$. Hence
		$$0=e=\lfloor-\alpha d\rfloor.$$
		The last equality gives $-\alpha d\ge 0$. Since $\alpha>0$ and $d\ge 0$, it follows that $d=0$. Hence $q=2p$ and $\{\alpha q\}=\{2\alpha p\}=0$, contrary to \eqref{eq7}. Therefore $r>0$ and $H_n\ge 1$. By \eqref{eq6} and \eqref{eq7}, we have $H_n=e$. Hence
		$$H_n=e=1.$$
		Therefore
		\begin{equation}\label{eq8}
			2r-\alpha d\ge 1.
		\end{equation}
		Now we prove that
		\begin{equation}\label{eq9}
			r>\frac pq.
		\end{equation}
		Firstly, we assume that $d>0$. Then $q>s_1$ and
		$$\alpha=\alpha_n\ge\alpha_1=\frac 1{s_1}>\frac 1q.$$
		By \eqref{eq8},
		$$r\ge\frac {1+\alpha d}{2}>\frac {1+d/q}{2}=\frac pq.$$
		Now we assume that $d=0$. Then $q=2p$ and \eqref{eq8} gives $r\ge 1/2$. If $r=1/2$, then
		$$\{\alpha q\}=\{2\alpha p\}=0,$$
		contrary to \eqref{eq7}. Thus \eqref{eq9} holds.
		
		Let
		\begin{equation}\label{eq10}
			h=\frac {1-r}{p},
			\quad \beta=\alpha+h.
		\end{equation}
		By \eqref{eq9},
		\begin{equation}\label{eq11}
			0<h<\frac 1p-\frac 1q.
		\end{equation}
		For $i<j$, we have $\lfloor\beta s_i\rfloor\ge\lfloor\alpha s_i\rfloor$. Thus, for $j\le n$, by \eqref{eq2} and \eqref{eq11},
		\begin{align*}
			\Phi_j(\beta)&\ge\Phi_j(\alpha)-hs_j\\
			&\ge\frac {s_j}{p}-hs_j>\frac {s_j}{q}.
		\end{align*}
		Also, $\beta p=a_n+1$, and hence $\lfloor\beta p\rfloor=\lfloor\alpha p\rfloor+1$. For $i<n$, $\lfloor\beta s_i\rfloor\ge\lfloor\alpha s_i\rfloor$. Therefore
		\begin{equation}\label{eq12}
			\Phi_{n+1}(\beta)\ge\Phi_{n+1}(\alpha)+1-qh.
		\end{equation}
		It follows from $e=\lfloor2r-\alpha d\rfloor=1$ that
		$$\{\alpha q\}=2r-\alpha d-1.$$
		In view of \eqref{eq1} and $H_{n+1}=0$, we have
		$$\Phi_{n+1}(\alpha)=1-\{\alpha q\}=2(1-r)+\alpha d.$$
		By $q\le 2p$ and $d\ge 0$, we have
		$$qh=\frac qp(1-r)\le 2(1-r)\le\Phi_{n+1}(\alpha).$$
		Thus \eqref{eq12} gives
		$$\Phi_{n+1}(\beta)\ge 1.$$
		Let $\alpha_{n+1}=\beta$. Then \eqref{eq2} holds with $n$ replaced by $n+1$. By \eqref{eq10} and \eqref{eq11},
		$$\alpha_{n+1}>\alpha_n,\quad
		\alpha_{n+1}+\frac 1q<\alpha_n+\frac 1p.$$
		Hence \eqref{eq3} holds.
		
		In view of \eqref{eq3}, we have
		$$\alpha_n\ge\alpha_1=\frac 1{s_1},\quad
		\alpha_n+\frac 1{s_n}\le\alpha_1+\frac 1{s_1}=\frac 2{s_1}.$$
		Thus \eqref{eq4} holds. It follows from \eqref{eq3} that $I_{n+1}\subset I_n$. For any $x\in I_n$, write $x=\alpha_n+u$, where $0\le u<1/s_n$. By Lemma \ref{lem2} and \eqref{eq2}, we have $\Phi_j(x)>0$ for $1\le j\le n$. In view of \eqref{eq1},
		$$\lfloor xs_j\rfloor\le 1+\sum_{i<j}\lfloor xs_i\rfloor\quad(2\le j\le n).$$
		By \eqref{eq4}, we have
		$$1\le xs_1<2,$$
		and hence $\lfloor xs_1\rfloor=1$. This completes the proof of Lemma \ref{lem3}.
	\end{proof}
	
	\begin{proof}[Proof of Theorem \ref{thm1}]
		Firstly, we assume that $s_1\ge 2$. Let $\alpha_n$ be as in Lemma \ref{lem3}, and let
		$$b_n=\alpha_n+\frac 1{s_n}.$$
		By \eqref{eq3}, the sequence $\{\alpha_n\}$ is nondecreasing and the sequence $\{b_n\}$ is strictly decreasing. By \eqref{eq4}, both sequences are bounded. Hence both sequences converge. Since $s_1<s_2<\cdots$, it follows that $s_n\to\infty$. Also,
		$$b_n-\alpha_n=\frac 1{s_n}\longrightarrow 0.$$
		Hence $\{\alpha_n\}$ and $\{b_n\}$ have the same limit, say $\alpha$. For any positive integer $N$, we have
		$$\alpha\ge\alpha_N,
		\quad \alpha\le b_{N+1}<b_N.$$
		Thus $\alpha\in I_N$. Let $j\ge2$ and take $N\ge j$. By Lemma \ref{lem3}, we have
		$$\lfloor\alpha s_j\rfloor\le 1+\sum_{i<j}\lfloor\alpha s_i\rfloor.$$
		Also, $\alpha\in I_1$, and hence $\lfloor\alpha s_1\rfloor=1$. It is clear that the terms of $S_\alpha$ are nondecreasing positive integers. By Lemma \ref{lem1}, we have
		$$P(S_\alpha)=\mathbb{N}_0.$$
		Furthermore,
		$$\alpha\ge\alpha_1=\frac 1{s_1},
		\quad \alpha\le b_2<b_1=\frac 2{s_1}\le 1.$$
		Thus $0<\alpha<1$.
		
		Now we assume that $s_1=1$. Then $1=s_1<s_2\le2s_1=2$, and hence $s_2=2$. By the first case, there exists a real number $\alpha$ with $1/2\le\alpha<1$ such that
		$$P((\lfloor\alpha s_2\rfloor,\lfloor\alpha s_3\rfloor,\cdots))=\mathbb{N}_0.$$
		Since $\lfloor\alpha s_1\rfloor=0$, it follows that
		$$P(S_\alpha)=\mathbb{N}_0.$$
		This completes the proof of Theorem \ref{thm1}.
	\end{proof}
	
	\begin{proof}[Proof of Theorem \ref{thm2}]
		Let $m$ be a positive integer such that $s_m\ge 2$ and
		$$s_{n+1}\le 2s_n\quad(n\ge m).$$
		By Theorem \ref{thm1}, there exists $0<\alpha<1$ such that
		$$P((\lfloor\alpha s_m\rfloor,\lfloor\alpha s_{m+1}\rfloor,\cdots))=\mathbb{N}_0.$$
		Every term of the sequence on the left is a term of $S_\alpha$. All terms of $S_\alpha$ are nonnegative. Hence
		$$\mathbb{N}_0\subseteq P(S_\alpha)\subseteq\mathbb{N}_0.$$
		Thus $P(S_\alpha)=\mathbb{N}_0$. This completes the proof of Theorem \ref{thm2}.
	\end{proof}

\end{document}